\documentclass{amsart} 
\usepackage{amssymb}
\usepackage{amsmath}
\usepackage{amsfonts}

\begin{document}
\newtheorem{theo}{Theorem}[section]
\newtheorem{prop}[theo]{Proposition}
\newtheorem{lemma}[theo]{Lemma}
\newtheorem{coro}[theo]{Corollary}
\theoremstyle{definition}
\newtheorem{exam}[theo]{Example}
\newtheorem{defi}[theo]{Definition}
\newtheorem{rem}[theo]{Remark}


\newcommand{\Bb}{{\bf B}}
\newcommand{\Cb}{{\bf C}}
\newcommand{\Nb}{{\bf N}}
\newcommand{\Qb}{{\bf Q}}
\newcommand{\Rb}{{\bf R}}
\newcommand{\Zb}{{\bf Z}}
\newcommand{\Ac}{{\mathcal A}}
\newcommand{\Bc}{{\mathcal B}}
\newcommand{\Cc}{{\mathcal C}}
\newcommand{\Dc}{{\mathcal D}}
\newcommand{\Fc}{{\mathcal F}}
\newcommand{\Ic}{{\mathcal I}}
\newcommand{\Jc}{{\mathcal J}}
\newcommand{\Kc}{{\mathcal K}}
\newcommand{\Lc}{{\mathcal L}}
\newcommand{\Mx}{{\mathcal M}}
\newcommand{\Nc}{{\mathcal N}}
\newcommand{\Oc}{{\mathcal O}}
\newcommand{\Pc}{{\mathcal P}}
\newcommand{\Qc}{{\mathcal Q}}
\newcommand{\Sc}{{\mathcal S}}
\newcommand{\Tc}{{\mathcal T}}
\newcommand{\Uc}{{\mathcal U}}
\newcommand{\Vc}{{\mathcal V}}
\newcommand{\btu}{\bigtriangleup}

\author{Charles Akemann, Simon Wassermann and Nik Weaver}

\title [Pure states on free group C*-algebras]
        {Pure states on free group C*-algebras}

\address {Department of Mathematics\\
           University of California\\
           Santa Barbara, CA 93106, USA}
\address {Department of Mathematics\\
           University of Glasgow\\
           Glasgow G12 9XF, UK}
\address {Department of Mathematics\\
           Washington University in Saint Louis\\
           Saint Louis, MO 63130, USA}

\email {akemann@math.ucsb.edu, asw@maths.gla.ac.uk,
nweaver@math.wustl.edu}

\date{\today}

\begin{abstract}
We prove that all the pure states of the reduced C*-algebra of a
free group on an uncountable set of generators are *-automorphism
equivalent and extract some consequences of this fact. (AMS
classification $46\mathrm{L}05$)
\end{abstract}

\maketitle

NOTATION: For any set $R$ with two or more elements, let $F_{R}$
denote the free group on $R$ with generators $\{u_r: r\in R\}$ and
let $C^*_r(F_{R})$ denote the reduced group C*-algebra. We won't
distinguish between the elements of  $F_{R}$ and the corresponding
unitary operators in $C^*_r(F_{R})$. In what follows, $r_0$ will
be a fixed element of $R$, $u$ will denote the generator $u_{r_0}$
and $F_u$ will denote the subgroup of $F_R$ generated by $u$. We
view $C^*_r(F_{S})$ as the C*-subalgebra of $C^*_r(F_{R})$
generated by the unitaries $\{u_s : s \in S\}$ and $C^*_r(F_u)$ as
the C*-subalgebra generated by $u$. Let $P_{u}$ (resp. $P_{S}$)
denote the unique trace preserving conditional expectation from
$C^*_r(F_{R})$ onto $C^*_r(F_u)$ (respectively $C^*_r(F_{S})$).
Recall that $C^*_r(F_u)$ is *-isomorphic to the *-algebra of
continuous complex valued functions on the unit circle, with $u$
going into the function $\theta(z)=z$. Let $f_{0}$ denote the
(unique!) pure state of $C^*_r(F_u)$ that satisfies $f_{0}(u)=1$
and let $f=f_{0}\circ P_{u}$.

\bigskip
If $R$ is uncountable, then $C^*_r(F_R)$ is inseparable. For
Card$(R) = \aleph_1$, the algebra $C^*_r(F_R)$ is discussed in
\cite[Cor.\ 6.7]{pop}, where it is shown that $C^*_r(F_R)$ is
inseparable, but that every abelian subalgebra is separable.
Powers \cite{pow} showed that for Card$(R) =2, C^*_r(F_{R})$ is
simple and has unique trace. Powers' method extends to general
$R$. For general free products of groups, simplicity and
uniqueness of trace follow by results of Avitzour \cite{av}. In
\cite{a} and \cite{al} the methods of \cite{ao} were used to
extend the simplicity and uniqueness of trace results to a host of
other group C*-algebras where free sets lurked in the underlying
groups. In \cite{arch} Rob Archbold also obtained related results.

\begin{lemma}\label{LM} If $S \subset R$,  Card $(S)>1$ and
$\alpha$  is a *-automorphism of $C^*_r(F_{S})$,  then $\alpha$
has an extension to a *-automorphism of $C^*_r(F_{R})$.

\end{lemma}

\begin{proof} Check that if $\alpha^{\prime}$ is defined
on the *-algebra $A$ generated by $C^*_r(F_{S})$ and the
generators in $R\backslash S$ by applying $\alpha$ to elements of
$C^*_r(F_{S})$ and leaving the other generators alone, then
$\alpha^{\prime}$ is a $*$-automorphism of $A$. Every element of
$C^*_r(F_{R})$ is representable in the form of an element of
$l^{2}(F_{R})$, and the $\mathrm{trace}$ of such an element is
simply the coefficient of the identity. Since the $\mathrm{trace}$
is unique on $C^*_r(F_{S})$ by \cite[Proposition 1]{a} (see also
\cite[3.1]{av}), $\alpha$ preserves the $\mathrm{trace}$. Thus it
is easy to verify that $\alpha^{\prime}$ preserves the
$\mathrm{trace}$ on $A$. Again by density of $A$ in
$l^{2}(F_{R})$, for any $a\in A$ and any $\epsilon>0$ there exists
$b\in A$ such that $||b||_{2}=1$ and $||ab||_{2}>||a||-\epsilon$.
So $||\alpha^{\prime}(b)||_{2}=||b||_{2}=1$ by invariance of the
$\mathrm{trace}$, and hence
$||\alpha^{\prime}(a)||\geq||\alpha^{\prime}(a)\alpha^{\prime}(b)||_{2}=||\alpha^{\prime}(ab)||_{2}=$
$||ab||_{2}>||a||-\epsilon$. Since a similar inequality holds for
$\alpha^{-1}$, we see that $\alpha^{\prime}$ extend by continuity
to an automorphism of $C^*_r(F_{R})$.
\end{proof}

\begin{lemma}\label{LM} The state $f$ is the unique state extension of
$f_{0}$ to $C^*_r(F_{R})$,  and $f$ is a pure state of
$C^*_r(F_{R})$. Moreover $f|_{C^*_r(F_{\mathrm{S}})}$  is pure for
any subset $S$ of $R$ that contains $r_0$.

\end{lemma}

\begin{proof} Let $g$ be a state of $C^*_r(F_{R})$ such that
$g(u)=1$. The Cauchy-Schwarz inequality applies to show that
$g((1-u)a)=g(a(1-u))=0$ for any $a\in C^*_r(F_{R})$. By induction,
$g(u^n)=g(u^{-n})=1$ for every natural number $n$. Fix $s\in
F_{R}\setminus F_{u}$. By the Cauchy-Schwarz inequality again, as
above, $g(u^nsu^{-n})=g(s)$ for every natural number $n$. Taking
$\xi$ to be the canonical trace vector in $l^2(F_R)$,
$l^2(F_R)=H_0\oplus H_1$, where $H_0$ is the closed linear span of
all vectors of form $w\xi$ with $w$ a reduced word in $F_R$ with a
non-zero power of $u$ on the left, and $H_1$ is the closed linear
span of those $w\xi$ with $w$ not ending in a non-zero power of
$u$ on the left. Then $u^nH_1\subset H_0$ for any non-zero integer
$n$ and $sH_0\subset H_1$. By \cite[Lemma 2.2]{choi} (see also
\cite[Lemma 3.0]{av})
   \[|g(s)| = \lim_{k\to\infty}|(1/k)\sum_{n=1}^k g(u^nsu^{-n})|
   \leq \lim_{k\to\infty}\|(1/k)\sum_{n=1}^k u^nsu^{-n}\|
   \leq \lim_{k\to\infty}\frac{2}{\sqrt{k}} = 0.\]
By linearity and continuity of $g$, this implies that $g
=g|_{C^*_r(F_{u})}\circ P_{u}$ and hence that $g =f$. An easy
convexity argument shows that $f$ is a pure state.

The conclusion of the last sentence of the Lemma follows
immediately from the conclusion of the first sentence.
\end{proof}

\begin{prop}\label{P}
Let $\{G_r\}_{r\in R}$ be a set of nontrivial countable groups and
for non-empty $S\subset R$, let $G_S$ be the free product
$(*_{r\in S}G_r)$. Given a nonempty countable subset $S_0$ of R,
if $g$ is a pure state on $C^*_r(G_R)$ there is a countable subset
$S$ of $R$ containing $S_0$ such that $g|_{C^*_r(G_S)}$ is a pure
state of $C^*_r(G_S)$. Moreover $C^*_r(G_S)$ is separable and also
simple if $|G_s|>2$ for some $s\in S$.
\end{prop}

\begin{proof} Assume without loss of generality that $R$ is
uncountable. For any non-empty countable $S\subset R$,
$C^*_r(G_S)$ is separable, and by \cite[3.1]{av} simple if
$|G_s|>2$ for some $s\in S$. If $(\pi_g, H_g, \xi_g)$ is the GNS
representation of $C^*_r(G_R)$ corresponding to $g$, sequences of
sets
   \[S_1\subset S_2\subset \ldots \subset R,\]
with each $S_i$ countably infinite, closed separable linear
subspaces
   \[{\Bbb C}\xi_g = H_1\subset H_2\subset \ldots \subset H_g\]
and, for each $i\geq 2$, a countable dense subset $X_i$ of the
unit sphere of $H_i$ such that
   \[X_2\subset X_3\subset \ldots\]
are constructed inductively so that
   \[\pi_g(C^*_r(F_{S_i}))H_i \subseteq H_{i+1}\]
for $i\geq 1$. Let $S_1$ be a non-empty countable subset of $R$
containing $S_0$. For the inductive step, given $S_i$ and $H_i$,
let $H_{i+1}$ be the closed linear span of
$\pi_g(C^*_r(G_{S_i}))H_i$, which is separable, and let $X_{i+1}$
be a countable dense subset of the unit sphere of $H_{i+1}$
containing $X_i$. By Kadison's transitivity theorem there is a
countable set ${\mathcal U}_{i+1}$ of unitaries in $C^*_r(G_R)$
such that for $\xi,\eta \in X_{i+1}$, $\pi_g(u)\xi = \eta$ for
some $u\in {\mathcal U}_{i+1}$. Since each such $u$ is a
norm-limit of a sequence of finite linear combinations of elements
of $G_R$, there is a countable subset $S'_{i+1}$ of $R$ such that
${\mathcal U}_{i+1}\subset C^*_r(G_{S'_{i+1}})$. Let $S_{i+1} =
S'_{i+1} \cup S_i$. Now let
   \[S=\bigcup_{i=1}^{\infty} S_i, \quad X = \bigcup X_i, \quad H
   = \overline{\bigcup_{i=1}^{\infty} H_i}.\]
Then $S$ and $X$ are countable, $H$ is separable,
$\pi_g(C^*_r(G_S))H \subseteq H$ and $X$ is dense in the unit
sphere of $H$. If $\xi, \eta \in X$, then $\pi_g(v)\xi = \eta$ for
some unitary  $v\in C^*_r(G_S)$. It follows that for any
$\varepsilon > 0$ and unit vectors $\xi, \eta \in H$,
$\|\pi_g(w)\xi - \eta\| < \varepsilon$ for some unitary $w\in
C^*_r(G_S)$, which implies that $\pi_g(C^*_r(G_S))|_H$ acts
irreducibly on $H$. Since $g|_{C^*_r(G_S)}$ is the state of
$C^*_r(G_S)$ corresponding to $\xi_g$, $g|_{C^*_r(G_S)}$ is pure.
\end{proof}

\begin{theo}\label{TH} Any two pure states of $C^*_r(F_{R})$ are
*-automorphism equivalent.

\end{theo}

\begin{proof} If $R$ is countable, the conclusion is immediate from
\cite{kos}. Assume that $R$ is uncountable. Let $g$ be a pure
state of $C^*_r(F_{R})$. We shall show that $g$ is *-automorphism
equivalent to $f$. By Proposition 0.3 there is a countably
infinite subset $S\subset R$ such that $r_0\in S$ and
$g|_{C^*_r(F_{\mathrm{S}})}$ is pure. We have already noted that
$C^*_r(F_{S})$ is simple, and it is obviously separable, so by
\cite{kos} choose a *-automorphism $\gamma_{0}$ of $C^*_r(F_{S})$
such that $g|_{C^*_r(F_S)}=\gamma_{0}^{*}(f|_{C^*_r(F_{S})})$. By
Lemma 0.1, extend $\gamma_{0}$ to a *-automorphism $\gamma$ of
$C^*_r(F_{R})$. We must show that $\gamma^{*}(f)=g$. Lemma 0.2
shows that $f|_{C^*_r(F_{S})}$ has unique state extension to
$C^*_r(F_{R})$. Since $\gamma$ is a *-automorphism extending
$\gamma_{0}$, the same uniqueness of state extension must follow
for $\gamma^{*}(f|_{C^*_r(F_{S})})=g|_{C^*_r(F_{S})}$. Thus
$\gamma^{*}(f)=g$.
\end{proof}

The next result is in contrast to Corollary 0.9 of \cite{cn3}.

\begin{theo}\label{TH} If $g$ is a pure state on $C^*_r(F_R)$, then its
hereditary kernel,
$$\{a \in C^*_r(F_R): g(a^*a+aa^*)=0\},$$ contains a sequential abelian
approximate unit, and hence a strictly positive element.
\end{theo}

\begin{proof} By Theorem 0.4 it suffices to prove this for $f$. Choose an
excising sequence $\{a_n\}$ for $f_0$ in $C^*_r(F_u)$, as defined
in \cite{excising}. Let $p = \lim a_n$ in $C^*_r(F_R)^{**}$. By
Lemma 0.2, $p$ is a minimal projection there.  By \cite[Prop.
2.2]{excising}, $\{a_n\}$ will excise $f$ and $\{1-a_n\}$ will be
an approximate unit for $\{a \in C^*_r(F_R): f(a^*a+aa^*)=0\},$ so
$\sum_i^\infty 2^{-n}(1-a_n)$ is strictly positive there.
\end{proof}

\begin{theo}\label{TH} Let $r_0 \in S \subset R$.

1. Any pure state of $C^*_r(F_S)$ has a unique extension to a pure
state of $C^*_r(F_R)$.

2. The projection $P_S$ is the unique conditional expectation of
$C^*_r(F_R)$ onto $\Cc^*(F_S)$.
\end{theo}

\begin{proof} 1. By Theorem 0.4, any pure state of $C^*_r(F_S)$ is
*-automorphism equivalent to $f|_{C^*_r(F_S)}$, and thus has the
unique extension property since Lemma 0.2 shows that
$f|_{C^*_r(F_S)}$ has that property.

2. If there were another conditional expectation $Q:C^*_r(F_R) \to
C^*_r(F_S)$ distinct from $P_S$, then the duals $Q^*$ and $P_S^*$
would have to be different on some element  of $C^*_r(F_S)^*$,
hence on some state of $C^*_r(F_S)$, hence on some pure state of
$C^*_r(F_S)$ by the Krein Milman Theorem \cite[p.\ 32]{kr}. This
is impossible by part 1 of this theorem.
\end{proof}

\bigskip
\noindent {\bf Concluding Remarks.} 1. A very similar proof to
that of Proposition 3 shows the related result that if $B$ is a
separable C*-subalgebra of an inseparable C*-algebra $A$, then if
$g$ is a pure state of $A$, there is a separable C*-subalgebra $C$
of $A$ such that $B\subseteq C$ and $g|_C$ is pure. An analogous
induction argument shows moreover that if $A$ is simple, then a
simple $C$ with these properties can be found.

\smallskip\noindent 2. The proof of Theorem 0.4 and the preceding lemmas
generalize in an obvious way to the general free product groups
$G_R = *_{r\in R}G_r$ considered in Proposition 0.3, provided that
one of the constituent groups $G_{r_0}$ is abelian with an element
of infinite order. Thus any two pure states of $C^*_r(G_R)$ are
*-automorphism equivalent. The corresponding generalizations of
Theorems 0.5 and 0.6 to these free product groups then follow,
with $G_{r_0}$ taking the place of $F_u$.

\end{document}